\documentclass[12pt]{article}
\usepackage{amsmath}
\usepackage{amssymb}
\usepackage{amsthm}
\usepackage[utf8]{inputenc}
\usepackage[english]{babel}
\usepackage{graphicx}
\usepackage{rotating}
\usepackage{array}
\newtheorem{theorem}{Theorem}[section]

\newtheorem{lemma}[theorem]{Lemma}
\newtheorem{conj}{Conjecture}
\newcommand\pfree{\mathop{\mbox{$(q-1)$-$\mathit{free}$}}}
\newcommand\ufree{\mathop{\mbox{$u$-$\mathit{free}$}}}

\newcommand\llfree{\mathop{\mbox{$l_1$-$\mathit{free}$}}}
\newcommand\lllfree{\mathop{\mbox{$l_2$-$\mathit{free}$}}}

\newcommand\F{\mathbb{F}_q}

\begin{document}
\title{Primitive Values  of Rational Functions at Primitive Elements of  a Finite Field }
\author{Stephen D. Cohen$^a$ \thanks{The first author is  Emeritus Professor of Number Theory, University of Glasgow.}, Hariom Sharma$^b$, Rajendra Sharma$^b$}
	
	\date{}
	\maketitle
	\begin{center}
		\textit{$^a$ 6 Bracken Road, Portlethen, Aberdeen AB12 4TA,  UK\\
                     $^b$ Department of Mathematics, Indian Institute of Technology Delhi, New Delhi, 110016, India}
	\end{center}
	\begin{abstract} Given a prime power $q$ and an integer $n\geq2$, we establish a sufficient condition for the existence of  a  primitive pair $(\alpha,f(\alpha))$  where  $\alpha \in \mathbb{F}_q$  and $f(x) \in \mathbb{F}_q(x)$  is a rational function of degree $n$. (Here  $f=f_1/f_2$, where $f_1, f_2$ are coprime polynomials of degree $n_1,n_2$, respectively, and $n_1+n_2=n$.)  For any $n$, such a pair is guaranteed to exist for sufficiently large $q$.  Indeed, when $n=2$,  such a pair definitely  does  {\em not} exist  only  for 28 values of $q$ and possibly (but unlikely) only for  at most $3911$ other values of $q$.
\end{abstract}
	
	\textbf{Keywords:} Finite Fields, Characters, Primitive element\\
	2010 Math. Sub. Classification: 12E20, 11T23
	\footnote{emails: Stephen.Cohen@glasgow.ac.uk (Stephen)\\
                     hariomsharma638@gmail.com (Hariom), rksharmaiitd@gmail.com (Rajendra)}
	
\section{Introduction}
Throughout this article let $q$ be a prime power and $n \ (\geq2)$ be a positive integer. We use $\mathbb{F}_{q}$ to  denote the  finite field of order $q$ and $\mathbb{F}_{q}^*$ for  the cyclic group of nonzero multiplicative elements of $\mathbb{F}_{q}$. A generator of the cyclic group $\mathbb{F}_{q}^*$ is called \textit{a primitive element} of $\mathbb{F}_{q}$.  For a rational function $f(x)\in \mathbb{F}_{q}(x)$ and $\alpha \in \mathbb{F}_q$, we call  a pair $(\alpha, f(\alpha))\in \mathbb{F}_{q} \times \mathbb{F}_{q}$  a \textit{primitive pair} in $\mathbb{F}_{q}$ if both $\alpha$ and $f(\alpha)$ are primitive elements in $\mathbb{F}_{q}$. Primitive elements have many  applications in cryptography, see \cite{PPR10}.  The security of many cryptographic schemes (e.g.,  Diffie-Hellmen key exchange and Elgamel encryption scheme) relies on the computational intractability of finding solutions to Discrete Logarithm Problem, which uses primitive elements as its fundamental tool.

Broadly, our aim in this article is to classify finite fields for which there exists a primitive pair in $\F$ for general rational functions $f(x) \in \F(x)$.  In order to make this more precise, we introduce some terminology and conventions.

First, to say that a polynomial $f(x)\in \F[x]$ has {\em degree} $n \geq 0$ we mean that  $f(x)=a_nx^n+ \cdots+a_0$, where $a_n \neq 0$; in particular, $f$ is non-zero.  Next, let $f(x)=f_1(x)/f_2(x)$ be a rational function in $\mathbb{F}_q(x)$, where $f_1, f_2$ are polynomials of degree $n_1,n_2$, respectively.  In our study, we always assume $f$  is expressed in its lowest terms, i.e., $f_1$ and $f_2$ are coprime in which case call the function $f$ as described an  {\em  $ (n_1,n_2)$-function} having whose   degree  is $\deg(f)=n_1+n_2$.
 Observe that$ (\alpha, f(\alpha))$ is a primitive pair if and only if  and only if $ (\alpha,(1/f)(\alpha))$ is a  primitive pair.  Hence, replacing $f$ by $1/f$, if necessary, we can suppose $n_1\geq n_2$.   Further, we can divide each of $f_1$ and $f_2$ by the leading coefficient of $f_2$ and suppose that $f_2$ is monic.

Finally, we introduce a minor restriction on the shape of $f$ to avoid some exceptional or awkward cases, namely,  we suppose that $f$ is {\em not exceptional}, i.e.,  {\em not} of the form  $f(x) = cx^jg^d(x)$, where $j$ is any integer (positive, negative or zero),   $d >1$  divides $q-1$ and $c\in \F^*$,  for any rational function $g(x) \in \F(x)$.  As way of explanation, we observe first that if $f(x)=g^d >1(x)$, where $d>1$ divides $q-1$, then $f(\alpha)$ necessarily is a $d$th power and therefore cannot be primitive.   Further, for example, if $f(x) =cxg^2(x)$, where $c$ is a non-square in $\F$, then, if $\alpha$ is primitive (and so a non-square), then $f(\alpha)$ is a square and so necessarily not primitive.

The question of the existence of primitive pairs has previously been considered in various cases of rational  functions.   For instance, Cohen \cite{PPR5} solved the existence problem for the specific (1,0)-function $x+1$ and, in \cite{PPR15}, Cohen et al. identified  all finite fields for which there exists a primitive pair for every standard (1,0)-function (i.e., linear polynomial).

Recently,  Booker et al. \cite{PPR2} classified all finite fields for  which  there exists a primitive pair for every (non-exceptional)  (2,0)-function, i.e., quadratic polynomials (not of the form $c(x+\beta)^2$ for  $c \in \F^*, \beta \in \F$).

Wang et al. \cite{PPR12} and Cohen \cite{PPR8} studied the existence problem for primitive pairs in respect of the specific (2,1)-function $ (x^2+1)/x$ for fields of even order and, more recently, Cohen et al. \cite{PPR14} ( Corollary 2) provided  a complete solution for the (2,1)-functions $(x^2 \pm1)/x$.

Anju and Sharma \cite{PPR3} supplied  a sufficient condition  for the existence of primitive pairs for the general (2,1)-function.  (See also \cite{PPR16}.)  Recently in \cite{PPR4}, Sharma, Ambrish and Anju  established a similar sufficient condition for the general (2,2)-function.

 In this paper, we take $f(x)$ to be a  general rational  function of degree $n$  and prove the existence of primitive pairs $(\alpha, f(\alpha))$ in $\mathbb{F}_{q}$  for sufficiently large prime powers $q$.   To make this more precise, for each positive integer $n$, let $R_n$ be the set of non-exceptional  rational functions $f= f_1/f_2,$  (with $f_1, f_2$  coprime and $f_2$  monic)   of degree $n$ (where $n=n_1+n_2$ and  $n_1 \geq n_2$) and define $Q_n$  as the set of prime powers
 $q$  such that, for each  $f \in R_n$, there exists a  primitive pair  $(\alpha, f(\alpha)), \alpha \in \F$.  For any positive integer define $W(m) = 2^{\omega(m)}$, where $\omega(m)$ is the number of distinct prime divisors of $m$. (Thus, $W(m)$ is the number of square-free divisors of $m$.)  The main theorem to be proved is the following:

\begin{theorem}\label{Main}

Let $n\geq 2 $ and  $q$ be a prime power. Suppose
\begin{equation} \label{cond}
q^{\frac{1}{2}}>nW(q-1)^2.
\end{equation}
Then   $q\in Q_{n}$.

Hence, for each $n \geq 2$, there exists $C_n>0$ such that, if $q>C_n$, then $q \in Q_n$.
\end{theorem}

Using a sieving modification of Theorem \ref{main} we also give explicit  values for $C_n,  n= 2, 3, 4, \text { and} 5$,  and conjecture that the best (least) value of $C_2$ is 311.

We remark that, for a specific rational function $f$ of degree $n$  (for example, if $f_1$ or $f_2$ is not square-free), one could reduce the factor $n$ on the right side of condition  (\ref{cond}) by an appropriate amount.

We defer a study of those exceptional rational functions $f$ for which there generally  exists a primitive pair  $(\alpha, f(\alpha))$ to another occasion.

\section{Preliminaries}
In this section, we state some related definitions and results required in the paper. For a divisor $u$ of $q-1$, an element $w\in \mathbb{F}_{q}^*$ is called $\ufree$, if $w=v^d$, where $v\in \mathbb{F}_{q}$ and $d|u$ implies $d=1$. Note that an element $w\in \mathbb{F}_{q}^*$ is $\pfree$ if and only if it is primitive.

 We refer \cite{PPR5} for basics on finite fields and characters of finite fields. Following Cohen and Huczynska  \cite{PPR6}, \cite{PPR7},  it can be shown that for each divisor $u$ of $q-1$ $$\rho_{u} : \alpha \mapsto \theta(u)\sum_{d|u}\frac{\mu(d)}{\phi(d)}\sum_{\chi_d}\chi_d(\alpha),$$ where $\theta(u)=\frac{\phi(u)}{u}~(\text{where }\phi \text{ is Euler's totient  function}),  \mu$ is M$\ddot{\text{o}}$bius function and $\chi_d$ denotes the multiplicative character of  $\mathbb{F}_{q}$  of order $d$, gives a characteristic function for the subset of $\ufree$ elements of $\mathbb{F}_{q}^*$.\\\\

 We shall need the following result of Weil \cite{PPR17}, as described in \cite{PPR18} at (1.2) and  (1.3),  for our main theorem.
 \begin{lemma}\label{CharBound}
 Let $F(x)\in \F(x)$ be a rational function. Write $F(x)= \prod_{j=1}^{k}F_j(x)^{r_j}$,
  where $F_j(x)\in \F[x]$ are irreducible polynomials and $r_j$ are non zero integers. Let $\chi$ be a multiplicative character of $\F$ of precise square-free order $d$ (a divisor of $q-1$). Suppose that $F(x)$ is not of the form $cG(x)^d$ for some rational function $G(x) \in \F(x))$ and $c \in \F^*$. Then we have $$\Big{|}\sum_{\alpha\in \mathbb{F}_{q}, F(\alpha) \neq \infty}\chi(F(\alpha))\Big{|}\leq \Big{(}\sum_{j=1}^k\deg(F_j)-1\Big{)}q^{\frac{1}{2}}.$$
\end{lemma}

A preliminary of another kind is subdivision of rational functions of degree $n$ into the union of $(n_1, n_2)-$ functions for every pairs $(n_1, n_2)$ with $n_1\geq n_2$ and $n_1+n_2=n$, as described in Section 1. Indeed, for each such pair $(n_1, n_2)$, define $R_{n_1, n_2}$ as the set of non exceptional $(n_1, n_2)$-rational functions, $Q_{n_1, n_2}$ as the set of prime powers $q$ such that for each $f\in R_{n_1, n_2}$ there exists a primitive pair $(\alpha, f(\alpha))$ and $C_{n_1, n_2}$ as a valid bound such that, if $q>C_{n_1, n_2}$, then $q\in Q_{n_1, n_2}$. Of course, our aim would be to find the least possible value for $C_{n_1, n_2}$ in every case, whence $C_n$ would be the maximum of the values of $C_{n_1, n_2}$ over the pairs $(n_1, n_2)$ with $n_1\geq n_2$ and $n_1+n_2=n$. More generally, for a set of rational functions $S$, define $Q_S$ and $C_S$ to say that $q>C_S$ implies $q\in Q_S$ in the above sense. For the present, simply observe the following. Suppose $f=f_1/f_2$ is a rational function with $n_1=n_2=n/2$. We always assume that $f_1$ and $f_2$ are coprime but supppose one of them is divisible by a positive power of $x$. In that case, the rational function $f^*(x)=f(1/x)$ written in its lowest terms has degree $n_0<n$. Moreover, since $\alpha$ a primitive element implies $1/\alpha$ is a primitive element, it follows that, if $(\alpha, f^*(\alpha))$ is a primitive pair, then $(1/\alpha$, $f(1/\alpha))$ is a primitive pair. Consequently, in effect, $f$ can be considered as having degree $n_0<n$ and therefore, when considering rational functions of degree $n$, if $n_1=n_2$, we can suppose that both $f_1$ and $f_2$ have non zero constatnt terms. For example, suppose $f(x)=a(x+b)/x, ab\neq0$ so that $f\in R_{1,1}$. Then $f^*(x)=ab(x+1/b)\in R_{1,0}$ and hence we can deduce $C_{\{f\}}=61$, form \cite{PPR15}.

\section{Sufficient conditions for the  existence of primitive pairs in $\mathbb{F}_{q}$}
 For each $m \in \mathbb{N}$, suppose $\omega(m)$ denotes the number of prime divisors of $m$ and $W(m)$ denotes the number of square free divisors of $m$. Let $l_1,l_2\in \mathbb{N}$ be such that if $l_1,l_2 \text{ divide }q-1$, then for each $f(x)\in R_n$, $N_{f}(l_1, l_2)$ denote the number of elements $\alpha\in \mathbb{F}_{q}$ such that $\alpha$ is $\llfree$ and $f(\alpha)$ is $\lllfree$. \\

 We now prove our one of the main results as follows.

\begin{theorem}\label{main}
Let $n\geq 2$, and $q$ be a prime power.  Suppose that
\begin{equation} \label{condition}
 q^{\frac{1}{2}}>nW(q-1)^2.
 \end{equation}
 Then $q\in Q_{n}$.
\end{theorem}

\begin{proof}
 To prove that $q\in Q_{n}$, we need to show that $N_{f}(q-1, q-1)>0$ for every (non-exceptional)   $f(x)\in R_n$. Now let $f(x)\in R_n$ be any   rational function.  Let  $S$ be the set of  poles of $f(x)$ in     $\mathbb{F}_{q}$.   Assume $q>2$  (as we may)  and  $l_1>1$ and $l_2>1$ are divisors of $q-1$.  Then by definition we have $$N_{f}(l_1,l_2)=\sum_{\alpha\in \mathbb{F}_{q} \setminus S} \rho_{l_1}(\alpha) \rho_{l_2}(f(\alpha))$$ and hence
  \begin{equation}\label{Nf}
 N_{f}(l_1,l_2)=\theta(l_1)\theta(l_2)\sum_{d_1|l_1, ~d_2|l_2}\frac{\mu(d_1)}{\phi(d_1)} \frac{\mu(d_2)}{\phi(d_2)} \sum_{\chi_{d_1},~\chi_{d_2}} \chi_{f}(\chi_{d_1}, \chi_{d_2}), \end{equation}
 where

 \begin{equation}\label{chif}
   \chi_{f}(\chi_{d_1}, \chi_{d_2})=\sum_{\alpha\in \mathbb{F}_{q}\setminus S}\chi_{d_1}(\alpha)\chi_{d_2}(f_{}(\alpha)).
 \end{equation}
 Let $d_1$ and $d_2$  be divisors of $q-1$ (not both 1) and $\chi_{d_1}$ and $\chi_{d_2}$ be specific characters of orders $d_1$, $d_2$, respectively.    In view of the M$\ddot{\text{o}}$bius functions in (\ref{Nf}) we can suppose that $d_1$ and $d_2$ are square-free.

 First suppose that $d_2=1$, i.e., $\chi_{d_2}=\chi_1$ is the trivial character.  Then $|\chi_f(\chi_{d_1}, \chi_1)| $ is at most  the sum of the number of zeros and poles of $f$ and so  does not exceed $n$.

 Accordingly, suppose $d_2>1$. Let $d$ be the least common multiple of $d_1$ and $d_2$, and so  a square-free divsor of $q-1$.    Moreover, $d/d_1 $ and $d_1$ are coprime, as are $d/d_2$ and $d_2$. Further,  there is a character $\chi_d$  of order $d$,   such that $\chi_{d_2}$ =$\chi_d^{d/d_2}$.   In that case,  $\chi_{d_1}= \chi_d^k$ for some integer $k$ with $0 \leq k<q-1$.

  From (\ref{chif})
$$\chi_{f}(\chi_{d_1}, \chi_{d_2})=\sum_{\alpha\in \mathbb{F}_{q}\setminus S} \chi_d(\alpha^{k}(f(\alpha)^{d/d_2})=\sum \limits_{\alpha\in \mathbb{F}_{q}\setminus S} \chi_d (F(\alpha)),$$
where $F(x)=x^k f^{d/d_2}(x)$.

Now write $f(x)=x^jf_0(x)$, where $j$ is  some integer (positive, negative or zero) and $f_0$ is a rational function such that $x$ divides neither the numerator nor denominator of $f_0(x)$.  Thus $F(x)= x^{k+\frac{jd}{d_2}}f_0^{d/d_2}(x)$  We can now apply Lemma \ref{CharBound} {\em unless} $f_0^{d/d_2}=c^{d/d_2}G^d$ for some rational function $G$
and $c \in \F$.  The latter, however, would imply that $f(x)=cx^jG^{d_2}(x)$, where we have assumed $d_2>1$, which would mean that $f$ is exceptional.  Since $f$ is not exceptional and the number of distinct zeros and poles of $F$ in an algebraic closure of $\F$  is at most $n+1$, we conclude from Lemma \ref{CharBound} that
\begin{equation} \label{bound}
  \Big{|} \chi_{f}(\chi_{d_1}, \chi_{d_2})\Big{|} \leq nq^\frac{1}{2}.
  \end{equation}

Of course, (\ref{bound}) holds when $d_2=1$ (and $d_1>1$). On the other hand, trivially,
\begin{equation} \label{bound2}
\chi_{f}(\chi_1, \chi_1) \geq q-1-(n+1).
\end{equation}
  Combining (\ref{bound}) and  (\ref{bound2}) in (\ref{chif}), we obtain
  \begin{eqnarray*}\label{Nf1}
  N_{f}(l_1, l_2 )  & \geq& \theta(l_1) \theta(l_2)\left\{(q-(n+1))-nq^{\frac{1}{2}}(W(l_1)W(l_2)-1)\right\} \nonumber\\
  &  >& \theta(l_1) \theta(l_2)\left\{q-nq^{\frac{1}{2}}W(l_1)W(l_2)\right\}.
  \end{eqnarray*}
  certainly, whenever  $q>nW(l_1)W(l_2)$.   It follows that, if $q>nW(l_1)W(l_2)$, then $ N_{f}(l_1, l_2 ) >0$.   In particular, the theorem follows by taking $l_1=l_2=q-1$.
\end{proof}

For further calculation work we shall need following results. Their proofs have been omitted as they follow on ideas from  \cite{PPR8} and \cite{PPR13}.
  \begin{lemma}\label{Wbound}
  \ For each $m\in \mathbb{N}$, $W(m)\leq c_mm^{\frac{1}{6}}$, where $c_m=\frac{2^s}{(p_1...p_s)^\frac{1}{6}}$, and $p_1,...,p_s$ are the  distinct primes less than $64$ which divide $m$.

  In particular, for all $m\in \mathbb{N}$, $c_m<37.469$, and for all odd $m$, $c_m<21.029$.
  \end{lemma}
  \begin{theorem}\label{sieve} Let $l|(q-1)$, and $\{p_1,...,p_s\} $ be the collection of all primes dividing $q-1$ but not $l$. Suppose $\delta=1-2\sum_{i=1}^{s}\frac{1}{p_i},~ \delta>0$ and $\Delta=\frac{(2s-1)}{\delta}+2$. If $q^{\frac{1}{2}}>n\Delta W(l)^2$ then $q\in Q_{n}$.
  \end{theorem}
 \section{Rational functions of degree 2}
From Section $1$, and the last paragraph of Section $2$, we can classify rational
functions of degree $2$ as either $(2, 0)$-functions, i.e., quadratic polynomials $ax^2+bx+c$, where $a(b^2-4ac)\neq 0$, or $(1, 1)$-functions with non-zero constant
terms, thus having the form $a(x+b)/(x+c), abc(b-c)\neq 0$. One can work with
both the cases simultaneously. But it is appropriate to recall that by a demanding
theoretical and computational analysis it has been established in \cite{PPR2}, that
$C_{2,0} = 211$ is a valid bound, and that this is the minimum possible. In this section,
we shall find an explicit (though non-optimal) value for $C_{1,1}$ and thereby one
for $C_2 $ by means of Theorems \ref{main} and \ref{sieve}. However, our argument assumes
merely that the functions we consider are in $R_2$ (rather than being restricted
to $R_{1,1})$.

Suppose that $q$ is a prime power and $n = 2$. From Lemma \ref{Wbound}, $ W(q -
1)\leq 37.469q^{\frac{1}{6}}$ so that $2W(q-1)^2<2807.852q^{\frac{1}{3}}$. Hence (\ref{condition}) holds whenever $q>4.901\times 10^{20}$ in  which case by Theorem \ref{main}, necessarily $q\in Q_2$. (Indeed when $q$ is even, by Lemma \ref{Wbound}, it suffices that $q>4.787\times 10^{17}$.) Now suppose $\omega(q-1)\geq 17$. Then $q\geq 2\times 3\times 5\times 7 \times \cdots \times 59> 1.9 \times 10^{21}$ so that $q\in Q_2$. (When $q$ is even and $\omega(q-1)\geq 15$, then $ q\geq 3\times 5\times 7 \times \cdots \times 53>1.6 \times 10^{19}$, so that $q\in Q_2.$)

We can therefore assume that $\omega(q-1)\leq16$ and $q\leq 4.901\times 10^{20}$. To make further progress, we use the sieving Theorem \ref{sieve} in place of Theorem \ref{main}. In Theorem \ref{sieve}, suppose $5\leq \omega(q-1) \leq16$ and take $l$ as the product of the least $5$ primes in $q-1.$ i.e. $W(l)=2^5$. Then $s\leq 11$ and $\delta$ will be at its least  value when
$\{p_1, p_2, \cdots, p_{11}\}=\{13, 17, \cdots, 53\}$, i.e. the set of primes from $6$th to $16$th. This yields $\delta>0.173170$ and $\Delta<123.267943 $, so that $2\Delta W(l)^2 <2.52453\times 10^{5} $. From Theorem \ref{sieve}, provided $q^{\frac{1}{2}}>2.52453\times 10^{5} $ i.e. $q> 6.3733\times 10^{10}$, then $q\in Q_2$. In fact, if $\omega(q-1)\geq 11$, then $q> 2\times 10^{11}$, which means we can assume $\omega (q-1)\leq 10.$

Now repeat this procedure using Theorem \ref{sieve} with $4\leq \omega(q-1)\leq 10$ and $W(l)=2^4$. Then $s\leq 6$, $\delta>0.2855034, \Delta< 40.5284367$, $2\Delta W(l)^2<20751$, whence $q\in Q_2$ provided $q>4.3061\times 10^8$ which is bound to be the case. But,  $w(q-1)\geq 10$, gives $q>6.46\times 10^9$. Hence the result holds for $\omega(q-1)=10$. Next, we assume $4\leq \omega(q-1)\leq 9$, take $W(l)=2^4$ so that $s\leq 5$, $\delta>0.3544689, \Delta< 27.3900959$, $2\Delta W(l)^2<14024$. Which proves the result for $\omega(q-1)=9$.

We apply the procedure when $3\leq \omega(q-1)\leq 8$ with limited success. Take $\omega(l)=3$ so that $s\leq 5$, $\delta>0.1557111$, $\Delta< 59.7993247$ and $2W(l)^2<7655$. Hence $q\in Q_2$ whenever $q>5.86\times 10^7$.

Finally, for $q<5.86 \times10^7$, we coded the criterion of Theorem \ref{sieve} and obtained an explicit list of $3937$ possible exceptions for which the criterion failed even when the exact prime factorization of $q-1$ was used (see  the appendix). The largest of these prime powers is $33093061$. We summarise these results for rational functions in $R_{1,1}$ in the next theorem.

\begin{theorem}\label{n2}
For rational functions in $f(x)=a(x+b)/(x+c)$, where $a, b, c\in \mathbb{F}_q^*$ with $b\neq c$, then the bound $33093061$ is a valid value for $C_{1,1}$.
\end{theorem}

Of course, the value of $C_{1,1}$ shown in Theorem \ref{n2} is not optimal. In the other direction we worked on the possible exceptions below 10000 computaionally in GAP \cite{GAP4} and obtained a list of true exceptions as follows:\\
\textbf{Case 1.}  ($f(x)\in R_{1,1}$)\\
$q = 3, 4, 5, 7, 9, 11, 13, 16, 19, 23, 25, 29, 31, 37, 41, 43, 49, 61, 67, 71,73, 79, 103, 121,\\
139, 151, 211$ and $331.$\\
\textbf{Case 2.}  ($f(x)\in R_{2,0}$)\\
$q = 3, 4, 5, 7, 11, 13, 19, 25, 31, 37, 41, 43, 61, 67, 71, 73, 79, 121, 151,$ and  $211.$

From \cite{PPR2}, we know that the above is a complete list of genuine exceptions in Case 2, and $C_{2,0}=211.$ Analogously, we propose the following conjecture.

\begin{conj}
	We have $C_{1,1}=331$ and the list of prime powers not in $Q_{1,1}$ is shown in Case 1, as above.
\end{conj}

We complete this section with some remarks on the set $S$ of exceptional quadratic polynomials, whose members comprise quadratics  of the form $f(x)=a(x+b)^2$, where $ab\neq 0$. In the context of Lemma \ref{CharBound} their irreducible part is of degree 1 and hence the condition of Theorem \ref{main} applies with $n=1.$ Here, if $(\alpha, f(\alpha))$ is primitive, then necessarily $a$ is a non square, in which case it suffices that $\alpha$ is primitive and $a(\alpha+b)^2$ is $L$-free, where $L$ is the odd part of $q-1$.  Denote by $R_{1^2,0}$  the subset of $S$ for which $a$ is a non-square.  By methods of this section this will lead to a better (smaller) lower bound for $C_{1^2,0}$ than the  one shown in Theorem 4.1 for $C_{1,1}$.
	\section{Case n=3, 4 and 5}
	In this section, we demonstrate how to get at least one value $C_n$ for each $n\in \mathbb{N}$ and $n\geq 2$. Further, we provide some calculated values to reduce the bound $C_n$ for $n=3, 4$ and $5$.

    As described above,  Theorem \ref{main} and Lemma \ref{Wbound} together imply that  if $n (37.469)^2 q^{\frac{1}{3}}<q^{\frac{1}{2}}$ then $q\in Q_n$ i.e.  $q>n^6(37.469)^{12}$ implies $q\in Q_n$. Hence, for each $n\in \mathbb{N}, n\geq 2$, one value of $C_n$ is $n^6(37.469)^{12} \approx n^6 \times 7.65713\times 10^{18}$.

    Thus $q> 5.583\times 10^{21}$, $q>3.137\times 10^{22}$ and $q>1.197\times 10^{23}$ imply $q\in Q_3$, $q\in Q_4$, and $q\in Q_5$, respectively. If $\omega(q-1)\geq 18$ then $q\geq 2\times 3\times 5 \times \cdots \times 61 >1.1728 \times 10^{23}$, and if $\omega(q-1)\geq 19$ then $q\geq 2\times 3\times 5 \times \cdots \times 67 > 7.858\times 10^{24} $. Hence $\omega(q-1)\geq 18$ implies  $q\in Q_3$, $q\in Q_4$, and  $\omega(q-1)\geq 19$ implies $q\in Q_5$. The repeated application of Theorem \ref{sieve} (as discussed above in the case $n=2$), with the values in Tables 1, 2 and 3, provide the bounds $C_3\approx 4.426\times 10^8, C_4\approx 7.867\times 10^8$, and $C_5\approx 1.23\times 10^9$, respectively.

	\begin{center}
		$$\text{Table 1}$$
		\begin{tabular}{|m{.6cm}|m{3cm}|m{.7cm}|m{2cm}|m{2.2cm}|m{2.3cm} |}
			\hline
			Sr. No.     &  $a \leq \omega(q-1)\leq b$ & $W(l)$   & $\delta>$   & $\Delta<$ & $3 \Delta W(l)^2$ $<$ \\
			\hline
			1       &  $a=5, b=17$ & $2^5$ & $0.1392719$  & $167.1445296$ & $513468$  \\
			2       & $a=4, b=11$  & $2^4$ & $0.2209872$   & $60.8269154$ & $46716$  \\
			
			3	    & $a=4, b=9$  & $2^4$ & $0.3544689$   & $27.3900959$ & $21036$\\
			
			\hline
		\end{tabular}
	\end{center}
	
	\newpage
	\begin{center}
		$$\text{Table 2}$$
		\begin{tabular}{|m{.6cm}|m{3cm}|m{.7cm}|m{2cm}|m{2.2cm}|m{2.3cm} |}
			\hline
			Sr. No.     &  $a \leq \omega(q-1)\leq b$ & $W(l)$   &  $\delta>$   & $\Delta<$ & $4 \Delta W(l)^2$ $<$ \\
			\hline
			1      &  $a=5, b=17$ & $2^5$ & $0.1392719$  & $167.1445296$ & $684624$  \\
			2      & $a=4, b=11$  & $2^4$ & $0.2209872$   & $60.8269154$ & $62287$  \\
			
			3	    & $a=4, b=9$  & $2^4$ & $0.3544689$   & $27.3900959$ & $28048$\\
			
			\hline
		\end{tabular}
	\end{center}
	\begin{center}
		$$\text{Table 3}$$
		\begin{tabular}{|m{.6cm}|m{3cm}|m{.7cm}|m{2cm}|m{2.2cm}|m{2.3cm} |}
			\hline
			Sr. No.    &  $a \leq \omega(q-1)\leq b$ & $W(l)$   &  $\delta>$   & $\Delta<$ & $5 \Delta W(l)^2$ $<$ \\
			\hline
			1      &  $a=5, b=18$ & $2^5$ & $0.1064850$  & $236.7747170$ & $1212287$  \\
			2      &$a=4, b=11$  & $2^4$ & $0.2209872$   & $60.8269154$ & $77859$  \\
			
			3	    & $a=4, b=9$  & $2^4$ & $0.3544689$   & $27.3900959$ & $35060$\\
			
			\hline
		\end{tabular}
	\end{center}

Note that, similar reduction can be done for each $n$.
All the results of this section can be summarized in the following theorem.
\begin{theorem}\label{cn}
	For each $n\in \mathbb{N}, n\geq 2$, one of the value for $C_n$ is $n^6\times 7.65713\times 10^{18}$. For $n=3, 4$ and $5$ it can be reduced to $4.426\times 10^8, 7.867\times 10^8$ and $1.23\times 10^9$ respectively.
\end{theorem}
Theorem \ref{main} and Theorem \ref{cn} together prove the main result of this article stated in Theorem \ref{Main}.

\section*{Acknowledgment}
Prof. R. K. Sharma is Consenys Blockchain Professor.   He wants to thank Consenys AG for the same.

            \bibliographystyle{plain}
	\bibliography{Bibliography}

\end{document}